\documentclass[12pt, letterpaper]{article}
\usepackage{amsfonts}
\usepackage{amssymb}
\usepackage{latexsym}  
\usepackage{pifont}  

\DeclareMathAlphabet{\mathbbold}{U}{bbold}{m}{n}
\newcommand\one{\mathbbold{1}}  
\newcommand\zero{\mathbbold{0}}  

\let\SavedRightarrow=\Rightarrow
\usepackage{marvosym}
\let\Rightarrow=\SavedRightarrow

\usepackage{amsmath}   

\usepackage{stmaryrd}  

{\begin{itemize} \setlength{\itemsep}{-1mm} %
\renewcommand\labelitemi{\ding{#1}}} %
{\end{itemize}}

\newtheorem{theorem}{Theorem}[section]
\newtheorem{definition}[theorem]{Definition}
\newtheorem{lemma}[theorem]{Lemma}
\newtheorem{corollary}[theorem]{Corollary}

\newtheorem{question}[theorem]{Question}
\newtheorem{notation}[theorem]{Notation}
\newtheorem{example}[theorem]{Example}

\newcommand\CC{\mathcal {C}}
\newcommand\SSS{\mathcal {S}}
\newcommand\FF{\mathcal {F}}
\newcommand\EE{\mathcal {E}}
\newcommand\GG{\mathcal {G}}
\newcommand\PP{\mathcal {P}}
\newcommand\RR{\mathcal {R}}

\newcommand\WW{\mathcal {W}}

\newcommand\BB{\mathcal {B}}

\newcommand\AAA{\mathcal {A}}

\newcommand\stsp{\mathrm{st}}   
\newcommand\ZFC{\mathrm{ZFC}}  
\newcommand\CH{\mathrm{CH}}  
\newcommand\GCH{\mathrm{GCH}}  
\newcommand\sa{\mathrm{sa}}   
\newcommand\Irrmm{\mathrm{Irr}_{\mathrm{mm}}}

\newcommand\hgt{\mathrm{ht}}   

\newcommand\bbbb{\mathfrak{b}}
\newcommand\dddd{\mathfrak{d}}

\newcommand\rrrr{\mathfrak{r}}
\newcommand\CCCC{\mathfrak{C}}
\newcommand\BBBB{\mathfrak{B}}
\newcommand\DDDD{\mathfrak{D}}

\newcommand\pre[2]{ {}^{#1} #2 }

\newcommand\sym{\mathrel{\Delta}} 

\newcommand\Fn{\mathrm{Fn}}      

\newcommand\cchi{{\raise 2 pt \hbox{$\chi$}}}

\newcommand\eop{{\Large \Coffeecup}}  

\newenvironment{proof}{{\bf Proof.}}{\eop\medskip}
\newenvironment{proofof}[1]{\medskip \textbf{Proof of #1.}}{\eop\medskip}

\begin{document}

\title{Irredundant Sets in Atomic Boolean Algebras}

\author{Kenneth Kunen\footnote{University of Wisconsin, 
Madison, WI  53706, U.S.A., \ \ kunen@math.wisc.edu} }

\maketitle

\begin{abstract}
Assuming $\GCH$, we construct an atomic boolean algebra
whose pi-weight is strictly less than the least size of a
maximal irredundant family.
\end{abstract}

\section{Introduction}
We begin by reviewing some standard notation regarding boolean algebras.
Koppelberg \cite{Kop} and Monk \cite{Monk_text, Monk_paper} contain
much more information.

\begin{notation}
In this paper, the three calligraphic letters $\AAA, \BB,\CC$
will always denote boolean algebras; in particular,
$\AAA$ will always denote a finite-cofinite algebra.
Other calligraphic letters denote subsets of boolean algebras.
$b'$ denotes the boolean complement of $b$.
$\BB \subseteq \CC$ means that $\BB$ is a sub-algebra of $\CC$,
and $\BB \subset \CC$ or
$\BB \subsetneqq \CC$ means that $\BB$ is a \emph{proper} sub-algebra of $\CC$.
Also, $\stsp(\BB)$ denotes the Stone space of $\BB$.
\end{notation}

The symbols $\subset$ and $\subsetneqq$ are synonymous,
but we shall use $\subset$ when the properness is obvious;
e.g., ``let $\BB \subset \PP(\omega)$ be a countable sub-algebra and
$\cdots\cdots$''.

Some further notation is borrowed either
from topology (giving properties of $\stsp(\BB)$) or
from forcing (regarding $\BB \backslash \{\zero\}$ as a forcing poset).

\begin{definition}
If $\SSS \subseteq \BB$, then $\SSS$ is \emph{dense in} $\BB$ iff
$\forall b \in \BB \backslash \{\zero\} \;
\exists d \in \SSS \backslash \{\zero\} \; [d \le b]$.
\end{definition}

\noindent
In forcing, we would say that
$\SSS \backslash \{\zero\}$ is dense in  $\BB \backslash \{\zero\}$.

\begin{definition}
The \emph{pi-weight}, $\pi(\BB)$, is the least size of
a dense subset $\SSS \subseteq \BB$.
\end{definition}

\noindent
This is the same as the topological notion $\pi(\stsp(\BB))$.

\begin{definition}
If  $\EE \subseteq \BB$,
then $\sa(\EE)$ is the sub-algebra of $\BB$ generated by $\EE$.
\end{definition}

\noindent
Note that $\sa(\emptyset) = \{\zero, \one\}$.
The notation $\langle \EE \rangle$ is more common in the literature, but we
shall frequently use the angle brackets to denote sequences.

If $\EE$ is a set of non-zero vectors in
a vector space, then $\EE$ is linearly 
independent iff no $a \in \EE$ is generated from the other elements of $\EE$;
equivalently, iff no non-trivial linear combination from $\EE$
is zero.  In boolean algebras,
these two notions are not equivalent, and are named, respectively,
``irredundance'' and ``independence'':

\begin{definition}
$\EE \subseteq \BB$ is \emph{irredundant} iff 
$a \notin \sa(\EE \backslash \{a\})$ for all $a \in \EE$.
\end{definition}

\begin{definition}
For $a \in \BB$:  $a^1 = a'$ and $a^0 = a$.
Then, $\EE \subseteq \BB$ is \emph{independent} iff 
for all $n \in \omega$ and all distinct $a_0, \ldots , a_{n-1} \in \EE$
and all $\epsilon \in \pre{n}{2}$,
$\bigwedge_{i < n} a_i^{\epsilon(i)} \ne \zero$.
\end{definition}

Some remarks that follow easily from the definitions:
Every independent set is irredundant.
If $\EE$ is a chain and $|\EE| \ge 2$ and $\zero, \one \notin \EE$,
then $\EE$ is irredundant but not independent.
No irredundant set can contain $\zero$ or $\one$
because $\zero,\one \in \sa(\GG)$ for every $\GG$,
even if $\GG = \emptyset$.

Irredundance and independence are similar in that they 
treat an element and its complement equivalently:

\begin{lemma}
\label{lemma-replace-comp}
Fix $\EE \subseteq \BB \backslash \{\zero, \one\}$.
If $b, b' \in \EE$, then $\EE$ is neither
irredundant nor independent.
If $b \in \EE$ and $b' \notin \EE$ and $\tilde \EE$ is obtained
from $\EE$ by replacing $b$ by $b'$,
then $\EE$ is irredundant iff $\tilde \EE$ is irredundant and
$\EE$ is independent iff $\tilde \EE$ is independent.
\end{lemma}

Monk \cite{Monk_paper} defines:

\begin{definition}
$\Irrmm(\BB)$ is the 
\emph{m}inimum size of a \emph{m}aximal irredundant subset of $\BB$.
\end{definition}

The following provides a simple way to prove 
maximal irredundance:

\begin{lemma}
\label{lemma-irred-generate}
If $\EE \subseteq \BB$ is irredundant and $\sa(\EE) = \BB$,
then $\EE$ is maximally irredundant in $\BB$.
\end{lemma}

The following provides a simple way to refute maximal irredundance.
It is attributed to McKenzie in 
Koppelberg \cite{Kop} (see Proposition 4.23):

\begin{lemma}
\label{lemma-irred-dense}
Assume that $\EE \subseteq \BB$, where
$\EE$ is irredundant.  Fix $d \in \BB \backslash \sa(\EE)$ such that
$\forall a \in \sa(\EE)\, [a \le d \to a = \zero]$.
Then $\EE \cup \{d\}$ is irredundant.
In particular, if 
$\EE$ is maximally irredundant in $\BB$, then
$\sa(\EE) $ is dense in $\BB $.
\end{lemma}
\begin{proof}
Assume that $\EE \cup \{d\}$ is not irredundant.
Then there are distinct
$a, a_1, \ldots a_n \in \EE$ such
that $a \in \sa \{ a_1, \ldots a_n , d \}$.
Then, fix $u,w \in \sa \{ a_1, \ldots a_n \}$ such that
$a = (u \wedge d) \vee (w \wedge d')$.
Now, $a \wedge d' = w \wedge d'$, so $a \sym w \le d$,
so $a \sym w = \zero$.  Then $a = w \in \sa \{ a_1, \ldots a_n \}$,
contradicting irredundance of $\EE$.
\end{proof}

\begin{corollary}
When $\BB$ is infinite,
$\pi(\BB) \le \Irrmm(\BB) \le |\BB| \le  2^{\pi(\BB)}$.
\end{corollary}
\begin{proof}
For the first $\le$:  If $\EE$  is maximally irredundant in $\BB$, then
$\sa(\EE)$ must be infinite (since it is dense in $\BB$), so 
$\pi(\BB) \le |\sa(\EE)| = |\EE| $.
\end{proof}

Note that Lemma \ref{lemma-irred-dense}
does not say that $\EE$ must be dense in $\BB$,
and the first $\le$ can fail for finite $\BB$.
For example, let $\BB = \PP(4)$ be the $16$ element boolean algebra.
If $\EE = \{a,b\}$ is an independent set
(e.g., $a = \{0,1\}$ and $b = \{1,2\}$),
then $\sa(\EE) = \BB$.
So, $\EE$ is a maximal irredundant set, showing that
$\Irrmm(\BB) = 2$, although $\pi(\BB) = 4$.
Also, let $\FF $ be the set of the four atoms (singletons).
Then $\sa(\FF) = \BB$.  So, $\FF$ is a maximal irredundant set.

Since there can be maximal irredundant sets of
different sizes in $\BB$, there is no simple notion of ``dimension''
as in vector spaces.
This phenomenon can occur in infinite $\BB$ as well:

\begin{example}
\label{ex-max-irr}
Let $\BB = \PP(\kappa)$, where $\kappa$ is any infinite cardinal.
Then $\BB$ has maximal irredundant set of size $2^\kappa$, but
$\pi(\BB) = \Irrmm(\BB) = \kappa$.
\end{example}
\begin{proof}
Following Hausdorff \cite{Hau},
let $\EE$ be an independent set of size $2^\kappa$;
then $\EE$ is irredundant and is contained in a maximal irredundant set.
To prove that $\Irrmm(\BB) = \kappa$:  As in 
\cite{Monk_paper}, let $\FF = \kappa\backslash \{0\}$;
that is, $\FF$ is the set of all proper
initial segments of $\kappa$.
$\FF$ is a chain, and hence irredundant.  To prove maximality,
fix $c \in \PP(\kappa) \backslash \sa(\FF)$; we
show that $\FF \cup \{c\}$ is not irredundant.  By Lemma
\ref{lemma-replace-comp}, WLOG $0 \in c$
(otherwise, replace $c$ by $c'$).  Let $\delta$ be the
least ordinal not in $c$.  Then $\delta , \delta + 1 \in \FF$
and $\delta = c \cap (\delta + 1)$, refuting irredundance.
\end{proof}

In view of examples like this, Monk \cite{Monk_paper} asks (Problem 1):

\begin{question}
\label{question_Irrmm}
Does $\Irrmm(\BB) = \pi(\BB)$ for every infinite $\BB$?
\end{question}

Assuming $\GCH$, the answer is ``no'':

\begin{theorem}
\label{thm-pi-irred}
If $2^{\aleph_1} = \aleph_2$, then there is an atomic boolean algebra
$\BB$ such that $\pi(\BB) = \aleph_1 < \Irrmm(\BB)$.
\end{theorem}

We do not know whether the hypothesis ``$2^{\aleph_1} = \aleph_2$''
can be eliminated here, although it can be weakened quite a bit,
as we shall see from the proof of Theorem \ref{thm-pi-irred}
in Section \ref{sec-proofs}.
This weakening (described in Theorem \ref{thm-pi-irred-better})
is expressed in terms of some cardinals,
such as $\bbbb_{\omega_1}$,  $\dddd_{\omega_1}$, etc.,
that are obtained by replacing $\omega$ by $\omega_1$ in
the definitions of the standard cardinal characteristics of the continuum,
such as $\bbbb$,  $\dddd$, etc.
These cardinals are discussed further in Section \ref{sec-small},
which also uses them to give some lower bounds to
the size of a $\BB$ that can possibly satisfy Theorem \ref{thm-pi-irred}.
Section~\ref{sec-remarks} contains some preliminary observations
on atomic boolean algebras.

\section{Remarks on Atomic Boolean Algebras}
\label{sec-remarks}

We are trying to find an atomic $\BB$ that answers
Monk's Question \ref{question_Irrmm} in the negative; that is,
such that $\pi(\BB) < \Irrmm(\BB)$.  Observe first:

\begin{lemma}
\label{lemma-atomic}
If $\BB$ is infinite and atomic and $\kappa = \pi(\BB)$,
then $\kappa$ is the number of atoms, and $\kappa$ is infinite,
and $\BB \cong \widetilde \BB$, where
$\AAA \subseteq \widetilde \BB \subseteq \PP(\kappa)$, and
$\AAA$ is the finite-cofinite algebra on $\kappa$.
\end{lemma}

So, we need only consider $\BB$ with
$\AAA \subseteq \BB \subseteq \PP(\kappa)$. 
The proof of Example \ref{ex-max-irr}
generalizes immediately to:

\begin{lemma}
\label{lemma-init-seg}
Assume that 
$\AAA \subseteq \BB \subseteq \PP(\kappa)$,
where $\kappa$ is any infinite cardinal
and $\AAA$ is the finite-cofinite algebra
and $\kappa \subseteq \BB$ 
\textup(i.e., $\BB$ contains all initial segments of $\kappa$\textup).
Then $\pi(\BB) = \Irrmm(\BB) = \kappa$.
\end{lemma}

When $\kappa = \omega$, $\BB$ must contain all initial segments,
so the two lemmas imply:

\begin{lemma}
If $\BB$ is atomic and $\pi(\BB) = \aleph_0$ then $\Irrmm(\BB) = \aleph_0$.
\end{lemma}

So, we shall focus here on obtaining
our $\BB$ with $\kappa = \omega_1$.
Then, note that in
Lemma \ref{lemma-init-seg}, a club of initial segments suffices:

\begin{lemma}
\label{lemma-contains-club}
Assume that 
$\AAA \subseteq \BB \subseteq \PP(\omega_1)$,
where $\AAA$ is the finite-cofinite algebra
and $C \subseteq \BB$  for some club $C \subseteq \omega_1$.
Then $\pi(\BB) = \Irrmm(\BB) = \aleph_1$.
\end{lemma}
\begin{proof}
Shrinking $C$ and adding in $0$ if necessary, we may assume
that $C$, in its increasing enumeration, is
$\{\delta_{\omega \cdot \alpha} : \alpha < \omega_1\}$,
where $\delta_0 = 0$ and 
$\delta_{\omega \cdot (\alpha + 1)} \ge
\delta_{\omega \cdot \alpha} + \omega$ for each $\alpha$.
Then each set-theoretic difference
$\delta_{\omega \cdot (\alpha + 1)} \backslash
\delta_{\omega \cdot \alpha}$ is countably infinite, so we can enumerate this
set as $\{\delta_{\omega \cdot \alpha + \ell} : 0 < \ell < \omega\}$.
We now have a 1-1 (but not increasing) enumeration of $\omega_1$
as $\{\delta_\xi : \xi < \omega_1\}$, and, using $\AAA \subseteq \BB$,
each initial segment
$\{\delta_\xi : \xi < \eta\} \in \BB$.
We can now apply Lemma \ref{lemma-init-seg} to the isomorphic
copy of $\BB$ obtained via the bijection $\xi \mapsto \delta_\xi$.
\end{proof}

We shall avoid this issue by constructing a $\BB$
with $\pi(\BB) < \Irrmm(\BB)$ so that $\BB$
contains no countably infinite sets at all;
we shall call such $\BB$ \emph{dichotomous}:

\begin{definition}
$\AAA$ always denotes the finite-cofinite algebra on $\omega_1$.
$\BB $ is \emph{dichotomous} iff
$\BB$ is a sub-algebra of $\PP(\omega_1)$ and
$\AAA \subseteq \BB$ and
$\forall b \in \BB\, [b \in \AAA \text{ or } |b| = |\omega_1 \backslash b| =
\aleph_1 ]$.
\end{definition}
Note that
$\forall b \in \BB\, [b \in \AAA \text{ or } |b| = |\omega_1 \backslash b| =
\aleph_1 ]$ is equivalent to 
$\forall b \in \BB\, [|b| \ne \aleph_0 ]$.

To get an easy example of a dichotomous $\BB$ of size $2^{\aleph_1}$:
Following Hausdorff \cite{Hau},
let the sets $J_\alpha \subset \omega_1$ for $\alpha < 2^{\aleph_1}$
be independent in the sense that all non-trivial boolean combinations
are uncountable (not just non-empty).  Then
$\BB = \sa(\AAA \cup \{J_\alpha : \alpha < 2^{\aleph_1}\} )$ is dichotomous.
However, it is quite possible that $\Irrmm(\BB) = \aleph_1$
because the following lemma may apply.
This goes in the opposite direction from
Lemma \ref{lemma-contains-club}:

\begin{lemma}
\label{lemma-nomax-split}
Assume that $\AAA \subseteq \BB \subset \PP(\omega_1)$,
and assume that $\omega_1 = \bigcup\{S_\xi : \xi < \omega_1\}$,
where the $S_\xi$ are disjoint countably infinite sets and
\[
\forall b \in \BB \backslash \AAA \; \exists \xi \;
[ S_\xi \cap b \ne \emptyset \;\&\; 
S_\xi \backslash b \ne \emptyset ] \ \ .
\tag{$\ast$}
\]
Then $\Irrmm(\BB) = \aleph_1$.
\end{lemma}

\begin{proof}
List each $S_\xi$ as $\{\sigma_\xi^\ell : \ell \in \omega \}$.
Then, let $\EE$ be the set of
all $\{\sigma_\xi^0, \ldots,  \sigma_\xi^\ell\}$
such that $\xi < \omega_1$ and $\ell < \omega$.
Then $\sa(\EE) = \AAA$, so $\EE$ is maximally irredundant in $\AAA$.

Also, $\EE$ remains maximal in $\BB$.  Proof:  fix 
$b \in \BB \backslash \AAA$ and then fix $\xi$ as in $(\ast)$.
WLOG $\sigma_\xi^0 \in b$ (otherwise, swap $b/b'$).
Then let $\ell$ be least such that
$\sigma_\xi^\ell \notin b$; so,
$\{\sigma_\xi^0, \ldots,   \sigma_\xi^{\ell-1}\} \subseteq b$.
Then $\{\sigma_\xi^0, \ldots,  \sigma_\xi^{\ell -1}\} = 
\{\sigma_\xi^0, \ldots,  \sigma_\xi^\ell\} \cap b$,
so $\EE \cup \{b\}$ is not irredundant.
\end{proof}

To see how this lemma might apply to 
$\BB = \sa(\AAA \cup \{J_\alpha : \alpha < 2^{\aleph_1}\} )$:
Start with any partition $\{S_\xi : \xi < \omega_1\}$.
Choose any $T_\xi$ with $\emptyset \subsetneqq T_\xi \subsetneqq S_\xi$.
Then, choose the independent $J_\alpha$ so that each
$J_\alpha \cap S_\xi$ is either $T_\xi$ or $\emptyset$.

Assuming that $2^{\aleph_1} = \aleph_2$,
our $\BB$ satisfying Theorem \ref{thm-pi-irred} will in fact
be dichotomous and of
the form $\sa(\AAA \cup \{J_\alpha : \alpha < \omega_2\} )$,
where the $J_\alpha$ are independent,
but the $J_\alpha$ will be chosen inductively,
in $\omega_2$ steps
to avoid situations such as the one described in Lemma \ref{lemma-nomax-split}.
The next section defines some cardinals below $2^{\aleph_1}$
that will be useful both in describing properties of clubs
and in deriving a version of Theorem \ref{thm-pi-irred}
that applies in some models of $2^{\aleph_1} > \aleph_2$.

\section{Some Small Cardinals}
\label{sec-small}

We begin with some remarks on club subsets of $\omega_1$.

\begin{definition}
\label{def-club-part}
Given a club $C \subseteq \omega_1$, we 
define the \emph{associated partition} of $\omega_1$ into $\aleph_1$
non-empty countable sets, which we shall call the 
\emph{$C$--blocks}, and label them
as $S_\xi^C$ \textup(or, just $S_\xi$\textup) for $\xi < \omega_1$.
If $0 \in C$, write $C = \{\gamma_\xi : \xi < \omega_1\}$ in increasing
enumeration;
then $S_\xi = [\gamma_\xi, \gamma_{\xi + 1})$.  If $0 \notin C$,
let $S_\xi^C = S_\xi^{C \cup \{0\}}$.
\end{definition}

Note that if we are given sets $S_\xi$ satisfying the hypotheses
of Lemma \ref{lemma-nomax-split}, then there is a club
$C$ such that each $S_\xi$ meets only one $C$--block.
Then, $\{S_\xi^C : \xi < \omega_1\}$ also will satisfy the hypotheses
of Lemma \ref{lemma-nomax-split}.

For our purposes, ``thinner'' clubs will yield ``better'' partitions.
As usual, for subsets of $\omega_1$,
$D \subseteq^* C$ means that $D \backslash C$ is countable.
Then observe

\begin{lemma}
\label{lemma-block-finer}
If $D \subseteq^* C \subseteq \omega_1$ and $D,C$ are clubs,
then all but countably many $D$--blocks are unions of $C$--blocks.
\end{lemma}

Given $\aleph_1$ clubs $C_\alpha$, for $\alpha < \omega_1$,
there is always a club $D$ such that $D \subseteq^* C_\alpha$ 
for all $\alpha$.
Whether this holds for more than $\aleph_1$ clubs
depends on the model of set theory one is in.
The basic properties here are controlled by the cardinals
$\bbbb_{\omega_1}$ and $\dddd_{\omega_1}$.

Cardinal characteristics of the continuum (e.g., $\bbbb$,  $\dddd$, etc.)
are well-known, and are discussed in set theory texts
(e.g., \cite{Je, Kunen3}), and in much more detail in the paper
of Blass \cite{Blass2}.
In analogy with $\bbbb$ and $\dddd$, we use $\bbbb_{\omega_1}$ to denote the
least size of an unbounded family in ${\omega_1}^{\omega_1}$,
while $\dddd_{\omega_1}$ denotes the least size of a dominating family.
Then $\bbbb_{\omega_1}$ is regular and
$\aleph_2 \le \bbbb_{\omega_1} \le \dddd_{\omega_1} \le  2^{\aleph_1}$.
Furthermore, statements such as
$\bbbb_{\omega_1} = \aleph_2$ and
$\bbbb_{\omega_1} = 2^{\aleph_1}$ and
$\aleph_2 < \bbbb_{\omega_1} < 2^{\aleph_1}$ are consistent with
$\CH$ plus $2^{\aleph_1}$ being arbitrarily large;
see \cite{Kunen3} \S V.5 for an exposition of these matters.
For our purposes here, it will often be useful to
rephrase $\bbbb_{\omega_1}$ and $\dddd_{\omega_1}$
in terms of clubs:

\begin{lemma}
Let $\CCCC$ be the set of all club subsets of $\omega_1$.
Then
$\dddd_{\omega_1}$ is the least $\kappa$ such that \textup(a\textup) holds and
$\bbbb_{\omega_1}$ is the least $\kappa$ such that \textup(b\textup) holds:
\[
\begin{array}{ll}
a. & \exists \DDDD \subseteq \CCCC  \;
  [ |\DDDD| = \kappa \;\&\; \forall C \in \CCCC \, \exists D \in \DDDD \,
  [ D \subseteq^* C ]]   \\
b. & \exists \BBBB \subseteq \CCCC  \;
  [ |\BBBB| = \kappa \;\&\; \neg \exists C \in \CCCC \, \forall D \in \BBBB \,
  [ C \subseteq^* D ]]    
\end{array} \ \ .
\]
\end{lemma}

The following definition relates clubs to the proof of
Lemma \ref{lemma-nomax-split}.
As before, $\AAA$ always denotes the
finite-cofinite algebra on $\omega_1$.

\begin{definition}
\label{def-induced-by}
A club $C \subset \omega_1$ is \emph{nice} iff all
$S_\xi = S^C_\xi$ are infinite.
If $C$ is nice, then
$\EE \subset \AAA$ is \emph{induced by} $C$ iff $\EE$
is obtained from the $S_\xi$ as in the proof of Lemma \ref{lemma-nomax-split}.
That is, list each $S_\xi$ as $\{\sigma_\xi^\ell : \ell \in \omega \}$;
then, $\EE = \{ \{\sigma_\xi^0, \ldots,  \sigma_\xi^\ell\} :
\xi < \omega_1 \,\&\, \ell < \omega \}$.
\end{definition}

Of course, $\EE$ is not uniquely defined from $C$, since 
$\EE$ depends on a choice of an enumeration of
each $S_\xi$.
Note that $\EE$ must be maximally irredundant in $\AAA$.  Whether
$\EE$ remains maximal in some $\BB \supseteq \AAA$ will depend on $\BB$.

For dichotomous $\BB$, the hypothesis $(\ast)$ of
Lemma \ref{lemma-nomax-split},
when $C$ is nice and $S_\xi = S^C_\xi$,
is equivalent to saying that
no $b \in \BB$ is \emph{blockish} for $C$:

\begin{definition}
\label{def-blockish}
If $C \subseteq \omega_1$ is a club, then
$b \subseteq \omega_1$  is \emph{blockish} for $C$
iff both $b$ and $b'$ are unions of $\aleph_1$ $C$--blocks.
\end{definition}

We next consider the
$\omega_1$ version of the \emph{reaping number} $\rrrr$.
This $\rrrr$ is less well-known than $\bbbb$ and $\dddd$, but
it is discussed in Blass \cite{Blass2}.

\begin{definition}
\label{def-reap}
If $\RR \subseteq [\omega_1]^{\aleph_1}$, then $T \subseteq \omega_1$
\emph{splits} $\RR$ iff 
$|X \cap T| = |  X \setminus T| = \aleph_1$ for all $X \in \RR$.
Then, $ \rrrr_{\omega_1}$ is the least cardinality of
an $\RR \subseteq [\omega_1]^{\aleph_1}$ such that
no $T \subseteq \omega_1$ splits $\RR$.
\end{definition}

Related to this, one might be tempted to
define a \emph{strong reaping number}:

\begin{definition}
\label{def-strong-reap}
If $\RR \subseteq [\omega_1]^{\aleph_1}$, then 
the nice \textup(Definition \ref{def-induced-by}\textup) club $C$
\emph{strongly splits} $\RR$ iff
every set $T$ that is blockish for $C$ splits $\RR$.
Then, $ \hat\rrrr_{\omega_1}$ is the least cardinality of
an $\RR \subseteq [\omega_1]^{\aleph_1}$ such that
no nice club strongly splits $\RR$.
\end{definition}

Some simple remarks:  The nice club $C$ strongly splits $\RR$
iff for each $X \in \RR$, all but countably many $C$--blocks meet $X$.
Also, if $C \subseteq^* D$ and 
$D$ strongly splits $\RR$, then $C$ strongly splits $\RR$.
Actually, $\hat\rrrr_{\omega_1} = \bbbb_{\omega_1}$
(although the concept of $\hat\rrrr_{\omega_1}$ will be useful); the
cardinals that we have defined are related by the following inequalities:

\begin{lemma}
\label{lemma-ineq}
$\aleph_2 \le \bbbb_{\omega_1} = \hat\rrrr_{\omega_1} \le \rrrr_{\omega_1}
\le 2^{\aleph_1}$
and
$\aleph_2 \le \bbbb_{\omega_1} = \hat\rrrr_{\omega_1} \le \dddd_{\omega_1}
\le 2^{\aleph_1}$.
\end{lemma}
\begin{proof}
For $\bbbb_{\omega_1} \le \hat\rrrr_{\omega_1}$:
For each $X \in \RR$, choose a nice club $C_X$ such that
all $C_X$--blocks meet $X$.
If $|\RR| < \bbbb_{\omega_1}$, then
there is a nice club $C$ such that $C \subseteq^* C_X$ for each $X \in \RR$.

For $\hat\rrrr_{\omega_1} \le \bbbb_{\omega_1}$:
Fix $\kappa < \hat\rrrr_{\omega_1} $; we shall show that 
$\kappa < \bbbb_{\omega_1}$.  So, let $C_\alpha$, for $\alpha < \kappa$
be clubs.  Then, fix a nice club $D$ that strongly splits
$\{C_\alpha: \alpha < \kappa\}$;
so, for each $\alpha$, all but countably many $D$--blocks meet $C_\alpha$.
But then $\tilde D \subseteq^* C_\alpha$ for each $\alpha$, where
$\tilde D$ is the set of limit points of $D$.
\end{proof}

It is not clear which of the many independence results involving
these cardinals on $\omega$ go through for the $\omega_1$ versions.
Of course, all these cardinals are $\aleph_2$ if
$2^{\aleph_1} = \aleph_2$.  Also, the following is easy by standard
forcing arguments:

\begin{lemma}
\label{lemma-forcing}
In $V$, assume that $\GCH$ holds and $\kappa > \aleph_2$ is regular.
Then there are cardinal-preserving
forcing extensions $V[G]$ satisfying each of the following:

1.  $\aleph_2 = \bbbb_{\omega_1} < 
 \dddd_{\omega_1} = \rrrr_{\omega_1} = 2^{\aleph_1} = \kappa$  .

2.  $\aleph_2 < \bbbb_{\omega_1} = 
 \dddd_{\omega_1} = \rrrr_{\omega_1} = 2^{\aleph_1} = \kappa$ .

3.  $\aleph_2 = \bbbb_{\omega_1} = 
 \dddd_{\omega_1} < \rrrr_{\omega_1} = 2^{\aleph_1} = \kappa$ .
\end{lemma}
\begin{proof}
For (1), use the ``countable Cohen'' forcing $\Fn_{\aleph_1}(\kappa, 2)$.
For (2), use a countable support iteration to get 
$V[G]$ satisfying Baumgartner's Axiom (see \cite{Kunen3}, \S V.5).

For (3), just use the standard Cohen forcing $\Fn(\kappa, 2)$.
Then in $V[G]$, $\rrrr_{\omega_1} = 2^{\aleph_1} = \kappa$ as in (1),
but $\bbbb_{\omega_1} = \dddd_{\omega_1} = \aleph_2$
because ccc forcing doesn't change $\bbbb_{\omega_1}$ or $ \dddd_{\omega_1}$.
\end{proof}

With this proof, $V[G] \models \CH$ in (1)(2), but
$V[G] \models 2^{\aleph_0} = \kappa$ in (3). 

We  do not know whether
$ \rrrr_{\omega_1} \ge \dddd_{\omega_1}$ is a $\ZFC$ theorem.
For the standard $\omega$ version,
$ \rrrr <  \dddd$ holds in the Miller real model
(see \cite{Blass2}, \S11.9), but it's
not clear how to make that construction work on $\omega_1$.
The following theorem will be proved in Section \ref{sec-proofs}:

\begin{theorem}
\label{thm-pi-irred-better} 
If $ \rrrr_{\omega_1} \ge \dddd_{\omega_1}$ then
there is a dichotomous $\BB$ such that
$\AAA \subset \BB \subset \PP(\omega_1)$ and
$|\BB| = \dddd_{\omega_1}$ and 
$\Irrmm(\BB) \ge \bbbb_{\omega_1}$.
\end{theorem}

Theorem \ref{thm-pi-irred} is an immediate consequence
of this, and if 
$ \rrrr_{\omega_1} \ge \dddd_{\omega_1}$ is a $\ZFC$ theorem,
then we would have a $\ZFC$ example answering 
Question \ref{question_Irrmm} in the negative.
The following shows that one cannot 
get $|\BB| < \dddd_{\omega_1}$ in Theorem \ref{thm-pi-irred-better}:

\begin{lemma}
\label{lemma-b-irr}
Assume that $\AAA \subseteq \BB \subset \PP(\omega_1)$
and $\BB$ is dichotomous and 
$|\BB| < \dddd_{\omega_1}$.  Then $\Irrmm(\BB) = \aleph_1$.
\end{lemma}
\begin{proof}
For each $b \in \BB \backslash \AAA$, let $f_b : \omega_1 \to \omega_1$
be such that for all $\xi$: $\xi < f_b(\xi)$ and
$b \cap (\xi, f_b(\xi)) \ne \emptyset$ and 
$b' \cap (\xi, f_b(\xi)) \ne \emptyset$.
Now, using $|\BB| < \dddd_{\omega_1}$, fix $g : \omega_1 \to \omega_1$
that is not dominated by the $f_b$;
that is, for all $b \in \BB \backslash \AAA$,
$Y_b := \{\xi : f_b(\xi) < g(\xi)\}$ is uncountable.
Let $C$ be a nice club of fixedpoints of $g$; that is, each $\delta \in C$
is a limit ordinal and $g(\xi) < \delta$ whenever $\xi < \delta$.

Now, fix $b \in \BB \backslash \AAA$.
Then fix $\xi \in Y_b$ with $\xi > \min(C)$.
So, we have $\xi < f_b(\xi) < g(\xi)$.
Let $\delta = \min\{\mu \in C : \mu \ge g(\xi)\}$.
Then $\xi < \delta \to g(\xi) < \delta$, so
$\xi < f_b(\xi) < g(\xi) < \delta$.
For each $\gamma < \delta$ such that $\gamma \in C$:
$\gamma < g(\xi)$ (by definition of $\delta$), so
$\gamma \le \xi$ (since $\xi < \gamma \to g(\xi) < \gamma$).
Fixing $\gamma = \sup(C \cap \delta)$, we have
$\gamma \le \xi < f_b(\xi) < g(\xi) < \delta$.
So, $\xi$ and $f_b(\xi)$ are in the same $C$--block, say
$S^C_\eta  = [\gamma, \delta)$, so by our choice of $f_b$,
$b \cap S^C_\eta \ne \emptyset$ and
$b' \cap S^C_\eta \ne \emptyset$.
Then $(\ast)$ of Lemma \ref{lemma-nomax-split} holds,
so $\Irrmm(\BB) = \aleph_1$.
\end{proof}

Assuming $\CH$, we can remove the hypothesis that $\BB$ is dichotomous:

\begin{lemma}
Assume $\CH$, and let $\BB$ be atomic with $\pi(\BB) = \aleph_1$ 
and $|\BB| < \dddd_{\omega_1}$.  Then $\Irrmm(\BB) = \aleph_1$.
\end{lemma}
\begin{proof}
WLOG, $\AAA \subseteq \AAA^* \subseteq \BB \subset \PP(\omega_1)$,
where $\AAA^*$ is the set of all $b \in \BB$
such that $b$ or $b'$ is countable.
Then $|\AAA^*| = \aleph_1$ by $\CH$.

Apply the proof of Lemma \ref{lemma-b-irr}, but just using
$f_b$ for $b \in \BB \backslash \AAA^*$.  This yields an $\EE \subset \AAA$
such that $\EE$ is maximally irredundant in $\AAA$ and
$\EE \cup \{b\}$ is not irredundant for all $b \in \BB \backslash \AAA^*$.
Let $\EE^* \subseteq \AAA^*$ be  maximally irredundant in $\AAA^*$
with $\EE^* \supseteq \EE$.  Then $\EE^*$ is  maximally irredundant in $\BB$
and $|\EE^*| = \aleph_1$.
\end{proof}

\section{A Very Blockish Boolean Algebra}
\label{sec-proofs}
Here we shall prove Theorem \ref{thm-pi-irred-better}.
Our only use of the assumption
$ \rrrr_{\omega_1} \ge \dddd_{\omega_1}$ will be to prove
Lemma \ref{lemma-dichot-reap} below.
First, a remark on preserving dichotomicity:

\begin{lemma}
\label{lemma-pres-dic}
Assume that $\AAA \subseteq \BB \subset \PP(\omega_1)$,
$\BB$ is dichotomous, and $b \subseteq \omega_1$.  Then TFAE:

1. $\sa(\BB \cup \{b\})$ is dichotomous.

2.
$|u \cap b| \ne \aleph_0$ and
$|u \cap b'| \ne \aleph_0$ for all $u \in \BB$.
\end{lemma}
\begin{proof}
$(1) \to (2)$ is immediate from the definition of ``dichotomous''.
Conversely, if (1) is false, fix
$a = (u \cap b) \cup (w \cap b') \in \sa( \BB \cup \{b\})$ such
that $|a| = \aleph_0$, where $u,w \in \BB$.  Then at least one of
$(u \cap b)$ and $ (w \cap b')$ has size $\aleph_0$,
so (2) is false.
\end{proof}

Note that (2) holds whenever $|b| = \aleph_1$ and 
$b$ splits $\BB \cap [\omega_1]^{\aleph_1}$.

\begin{lemma}
\label{lemma-dichot-reap}
Assume that 
$\kappa := \dddd_{\omega_1} \le  \rrrr_{\omega_1}$. 
Then there is a $\BB$ such that
$\AAA \subset \BB \subset \PP(\omega_1)$,
and $\BB$ is dichotomous,
and $|\BB| = \kappa$, and
\[
\text{
For all clubs $C \subset \omega_1$ 
there is a $b \in \BB$ such that $b$ is blockish for $C$.
}
\tag{\dag}
\]
\end{lemma}
\begin{proof}
Let $C_\mu \subset \omega_1 $ for $\mu < \kappa$ be
nice (Definition \ref{def-induced-by}) clubs such that 
for every club $C\subset \omega_1 $  there is a $\mu$ with $C_\mu \subseteq C$.

Now, build a chain $\langle \BB_\mu : \mu \le \kappa \rangle$, 
where $\AAA \subseteq \BB_\mu \subseteq \PP(\omega_1)$ and
$\mu \le \nu \to \BB_\mu \subseteq \BB_\nu$ and
all $\BB_\mu$ are dichotomous and $|\BB_\mu|  = \max( |\mu|, \aleph_1)$.
Let $\BB_0 = \AAA$, and take unions at limits.
Choose $\BB_{\mu+1} \supseteq \BB_\mu$ so that
$\BB_{\mu+1} = \sa(\BB_\mu \cup \{J_\mu\})$, where $J_\mu$ is blockish
for $C_\mu$.
Assuming that this can be done, setting $\BB = \BB_\kappa$
satisfies the lemma.

Fix $\mu$;  we show that an appropriate $J = J_\mu$ can be chosen:
$J$ will be blockish for $C_\mu$ and 
$|u \cap J| = |u \cap J'| = \aleph_1$ for all infinite (= uncountable)
$u \in \BB_\mu$.  Then, we can simply apply
Lemma \ref{lemma-pres-dic}.

Let $S_\xi = S_\xi^{C_\mu}$; these sets are disjoint and countably infinite.
For $u \in \BB_\mu \cap [\omega_1]^{\aleph_1}$,
let $\hat u = \{\xi < \omega_1 : u \cap S_\xi \ne \emptyset\}$.
Then $|\hat u| = \aleph_1$.
Since $| \BB_\mu  | <  \rrrr_{\omega_1}$, fix $T \subset \omega_1$
such that $T$ splits $\{\hat u : u \in \BB_\mu \cap [\omega_1]^{\aleph_1} \}$.
Then, let $J = \bigcup \{S_\xi : \xi \in T\}$.
Then $|u \cap J| = |u \cap J'| = \aleph_1$ for
all $u \in \BB_\mu \cap [\omega_1]^{\aleph_1}$.
\end{proof}

Lemma \ref{lemma-dich-block} below shows
(in $\ZFC$) that any $\BB$ satisfying $(\dag)$
also satisfies Theorem \ref{thm-pi-irred-better} --- that is,
$\Irrmm(\BB) \ge \bbbb_{\omega_1}$.
We remark that $(\dag)$ implies that for all clubs $C$,
the $S_\xi^C$ fail to satisfy $(\ast)$ of Lemma \ref{lemma-nomax-split}.
But that alone proves nothing, since possibly $\Irrmm(\BB) = \aleph_1$
via some $\EE$ that is not at all related to families
induced by clubs (Definition \ref{def-induced-by}).
But our argument will in fact show (Lemma \ref{lemma-dichot})
that such families
are all that we need to consider.

We remark that the $J_\mu$ used in the proof of
Lemma \ref{lemma-dichot-reap} are independent in the sense 
that all non-trivial finite boolean combinations are uncountable;
this is easily proved using the fact that
$|u \cap J_\mu| = |u \setminus J_\mu| = \aleph_1$ for all infinite
$u \in \sa \{J_\nu : \nu < \mu\}$.  But, as remarked above
(see end of Section \ref{sec-remarks}), 
independence alone is not enough to prove $\Irrmm(\BB) > \aleph_1$.

\begin{lemma}
\label{lemma-block-irred}
Assume that $\AAA \subseteq \BB \subset \PP(\omega_1)$
and $\BB$ is dichotomous and $|\BB| < \bbbb_{\omega_1}$.
In addition, assume that $\EE \subseteq \BB$
and $\EE$ is irredundant. 
Then there is a nice
club $D$ such that for any $c$ that is blockish for $D$:
\[
\begin{array}{l}
1.\; |c \cap b | = |c' \cap b| = \aleph_1 \text{ for all infinite } b \in \BB.
\\
2.\; c \notin \BB. \qquad
3.\; \sa(\BB \cup \{c\}) \text{ is dichotomous. } \qquad
4.\; \EE \cup \{c\} \text{ is irredundant. }
\end{array}
\]
\end{lemma}
\begin{proof}
Using $|\BB| < \bbbb_{\omega_1} = \hat\rrrr_{\omega_1}$
(Lemma \ref{lemma-ineq}),
fix a nice club $C$ such that
(1) holds for every $c$ that is blockish for $C$.
Then, for such $c$, (2) holds (setting $b = c$) and (3)
holds by Lemma \ref{lemma-pres-dic}.
Now, we cannot simply let $D = C$, since we have not used $\EE$
yet; for example, it is quite possible that
$\EE$ contains some $\{\alpha\}$ and $\{\alpha, \beta\}$
and there is a $c$ that is blockish for $C$ such that
$c \cap \{\alpha, \beta\} =\{\alpha\}$, so that
$\EE \cup \{c\}$ is not irredundant.
But, our proof will replace $C$ 
by a thinner club $D$
obtained via a chain of elementary submodels.

We recall some standard terminology on elementary submodels,
following the exposition in \cite{Kunen3} \S III.8:
Fix a suitably large regular $\theta$.  Then, a
\emph{nice chain of elementary submodels of $H(\theta)$}
is a sequence $\langle M_\xi :  \xi < \omega_1 \rangle$ such that
$M_0 = \emptyset$, and $M_\xi \prec H(\theta)$ for $\xi \ne 0$, and
all $M_\xi$ are countable, and
$\xi < \eta \to M_\xi \in M_\eta \; \& \; M_\xi \subset M_\eta$,
and $M_\eta = \bigcup_{\xi < \eta} M_\xi$ for limit $\eta$.
For $x \in \bigcup_\xi M_\xi$,
$\hgt(x)$  (the \emph{height} of $x$) denotes 
the $\xi$ such that $x \in M_{\xi+1} \backslash M_\xi$.
Given such a chain, 
let $\gamma_\xi = M_\xi \cap \omega_1 \in \omega_1$; so $\gamma_0 = 0$.
The \emph{associated club} is $D = \{\gamma_\xi :  \xi < \omega_1\}$.
If $S_\xi = [\gamma_\xi, \gamma_{\xi + 1}) =
\{\delta : \hgt(\delta) = \xi\}$,
then these $S_\xi$ are precisely the $S_\xi^D$ described in
Definition \ref{def-club-part}.

We shall use such a chain, with $C \in M_1$.  This will ensure
that $D \subset C \cup \{0\}$.  We also assume that $\EE \in M_1$.

Let $c$ be blockish for $D$ (and hence for $C$).
Then  $c \notin \BB$, so $c \notin \EE$.
WLOG, $\EE$ is maximally irredundant in $\BB$; if not, we
can replace $\EE$ by some maximally irredundant
$\tilde \EE \supset \EE$ such that $\tilde \EE \in M_1$.

Before proving irredundance of $\EE \cup \{c\}$, we introduce some
notation.  For each $\delta \in \omega_1$ and each $e \in \EE$,
let $h(\delta, e)$ be the \emph{smallest finite}
$r \in \sa(\EE \backslash \{e\})$
such that $\delta \in r$; if there is no such finite $r$,
let $h(\delta, e) = \infty$.
Maximality of $\EE$ plus Lemma \ref{lemma-irred-dense} implies
that $\{\delta\} \in \sa(\EE)$, and hence
$\{\delta\} \in \sa(\WW)$ for some finite $\WW \subset \EE$.
Then $h(\delta, e) = \{\delta\} \ne \infty$ for all $e \in \EE \backslash \WW$.
Observe that
\[
h(\delta, e) \ne \infty 
\; \& \; 
h(\varepsilon, e) \ne \infty 
\; \rightarrow \; \;
h(\delta, e) = h(\varepsilon, e)  \text{ or }
h(\delta, e) \cap h(\varepsilon, e) = \emptyset \ \ .
\tag{$\ast$}
\]
To prove $(\ast)$, use the definition of $h(\delta, e)$ as 
``the \emph{smallest} $r$ $\cdots$'': 
If $\delta \notin h(\varepsilon, e)$,
then $h(\delta, e) \cap h(\varepsilon, e) = \emptyset$
(otherwise one could replace $h(\delta, e) $ by the smaller 
$h(\delta, e) \setminus h(\varepsilon, e)$\,).
If $\delta \in h(\varepsilon, e)$ and $\varepsilon \in h(\delta, e)$,
then $h(\delta, e) = h(\varepsilon, e)$
(otherwise one could replace both
$h(\delta, e)$ and $h(\varepsilon, e)$ by the smaller 
$h(\delta, e) \cap h(\varepsilon, e)$\,).

Now, assume that $\EE \cup \{c\}$ is not irredundant.
Then, fix $a \in \EE$ such that $a \in \sa((\EE \setminus \{a\})\cup \{c\})$.
Then, fix $u,w \in \sa(\EE \backslash \{a\})$ such that
$a = (u \cap c ) \cup (w \cap c')$.
Then $u \cap w \subseteq a \subseteq u \cup w$.
Let $s = ( u \cup w ) \setminus (u \cap w ) = u \sym w $. 
Then $s \in \sa(\EE \backslash \{a\})$.
Note that $s$ is finite. To prove this, use 
(1) four times, plus the fact that $u,w,a \in \BB$:

$(w \backslash a) \cap c' = \emptyset$
so $w \backslash a $ is finite.

$(u \backslash a) \cap c = \emptyset$
so $u \backslash a $ is finite.

$((w \backslash u) \setminus (w \backslash a)) \cap c = \emptyset$
so 
$(w \backslash u) \setminus (w \backslash a)$ is finite
so 
$(w \backslash u)$ is finite

$((u \backslash w) \setminus (u \backslash a)) \cap c' = \emptyset$
so 
$(u \backslash w) \setminus (u \backslash a)$ is finite
so 
$(u \backslash w)$ is finite

\noindent
Then, $s = (w \backslash u) \cup (u \backslash w)$ is finite.

For $\delta \in s$, $h(\delta, a) \ne \infty$
because $h(\delta, a)  \subseteq s$.
For $\delta \in s$ and $\xi < \omega_1$,
$\delta \in M_\xi \leftrightarrow h(\delta, a) \in M_\xi$
(hence $\hgt( h(\delta, a) ) = \hgt(\delta) $).  Proof:
The $\leftarrow$ direction is clear because $\delta \in  h(\delta, a)$
and $h(\delta, a)$ is finite.
For the $\rightarrow$ direction, there are two cases:
If $a \in M_\xi$, use $M_\xi \prec H(\theta)$.
If $a \notin M_\xi$, use $\EE \in M_\xi \prec H(\theta)$ to get
a finite $\WW \in M_\xi$ such that
$\WW \subset \EE$ and $\{\delta\} \in \sa(\WW)$.
Then $a \notin \WW$ so $h(\delta, a) = \{\delta\}$.

Let $t = \bigcup \{ h(\delta, a) : \delta \in s \cap c \}$.
Then $s \cap c \subseteq t \subseteq s$.  
Also, this is a finite union, so $t \in \sa(\EE \backslash \{a\})$.
But actually, $t = s \cap c$:  If this fails, then
fix $\varepsilon \in t \backslash c$.
Then $\varepsilon \in  h(\delta, a) $ for some $\delta \in s \cap c$.
Applying $(\ast)$, $h(\delta, a) = h(\varepsilon, a)$, and so
$\hgt(\delta ) = \hgt( h(\delta, a) ) = 
\hgt(  h(\varepsilon, a) ) = \hgt(\varepsilon)$.
But $\delta \in c$ and $\varepsilon \notin c$, so this
contradicts the fact that $c$ is blockish.

So, $s \cap c \in \sa(\EE \backslash \{a\})$, and hence also
$s \cap c' \in \sa(\EE \backslash \{a\})$.
But then $a = (u \cap w) \cup (u \cap s \cap c) \cup (w \cap s \cap c')
\in \sa(\EE \backslash \{a\})$, contradicting irredundance of $\EE$.
\end{proof}

For dichotomous $\BB$,
we have a dichotomy for $\Irrmm(\BB)$; either
$\Irrmm(\BB) = \aleph_1$ or $\Irrmm(\BB) \ge \bbbb_{\omega_1}$:

\begin{lemma}
\label{lemma-dichot}
Assume that
$\AAA \subseteq \BB \subset \PP(\omega_1)$
and $\BB$ is dichotomous and $\Irrmm(\BB) < \bbbb_{\omega_1}$.
Then $\Irrmm(\BB) = \aleph_1$.
Furthermore, there is a nice club $D \subset \omega_1$ such that 
every $\EE \subset \AAA$ that is induced by $D$
is maximally irredundant in $\BB$,
and no $c \in \BB$ is blockish for $D$.
\end{lemma}
\begin{proof}
Fix  $\FF \subset \BB$ such that $|\FF| < \bbbb_{\omega_1}$
and $\FF$ is maximally irredundant in $\BB$.
Let $\widetilde \BB = \sa(\FF)$.
Then $\AAA \subseteq \widetilde \BB \subseteq \BB \subset \PP(\omega_1)$
and $| \widetilde \BB | < \bbbb_{\omega_1}$.
Apply Lemma  \ref{lemma-block-irred} to $\FF$ and $\widetilde \BB$.
This produces a nice club $D$ such that for any $c$ that is blockish for $D$:
$c \notin \widetilde \BB$ and $\FF \cup \{c\}$ is irredundant.

Now, consider any $c \in \BB \backslash \AAA$.  By maximality 
of $\FF$ in $\BB$, $c$ is not blockish for $D$.
Since $|c| = |c'| = \aleph_1$, there must be some $\xi$
such that $S^D_\xi \cap c \ne \emptyset$ and $S^D_\xi \cap c' \ne \emptyset$;
this is condition $(\ast)$ of
Lemma \ref{lemma-nomax-split}.  Then the proof of that lemma
shows that every $\EE \subset \AAA$ that
is induced by $D$ is maximally irredundant in $\BB$.
\end{proof}

It is important here that $\BB$ be dichotomous.
Otherwise, let $\BB = \PP(\omega_1)$.  Then $\Irrmm(\BB) = \aleph_1$
(Example \ref{ex-max-irr}), but no
$\EE \subseteq \AAA$ is maximally irredundant in $\BB$
(Lemma \ref{lemma-block-irred}, applied with $\BB = \AAA$).

Lemma \ref{lemma-dichot} implies immediately:

\begin{lemma}
\label{lemma-dich-block}
Assume that $\AAA \subset \BB \subset \PP(\omega_1)$,
and $\BB$ is dichotomous, and for all clubs $C \subset \omega_1$ 
there is some $b \in \BB$ such that $b$ is blockish for $C$.
Then $\Irrmm(\BB) \ge \bbbb_{\omega_1}$.
\end{lemma}

\begin{proofof}{Theorems \ref{thm-pi-irred-better} and \ref{thm-pi-irred}}
For Theorem \ref{thm-pi-irred-better}, apply
Lemma \ref{lemma-dich-block} to the $\BB$ obtained in Lemma
\ref{lemma-dichot-reap}.
Then Theorem \ref{thm-pi-irred} is the special case of
Theorem \ref{thm-pi-irred-better} where $2^{\aleph_1} = \aleph_2$.
\end{proofof}

\end{document}